\documentclass[12pt,reqno]{amsart}
\usepackage[english]{babel}

\usepackage{mathrsfs}
\usepackage{amsthm}
\usepackage{amssymb}
\usepackage{amsmath,amsthm,amsfonts}
\usepackage{mathabx}
\usepackage{hhline}
\usepackage{textcomp}
\usepackage{amsmath,amscd}
\usepackage[all,cmtip]{xy}
\usepackage{mathtools}
\usepackage[all]{xy}
\usepackage[margin=1in]{geometry}

\author{Kaveh Eftekharinasab}
\title{Geometry of bounded Fr\'{e}chet manifolds}
\address{Topology dept. \\ Institute of Mathematics of NAS of Ukraine \\ Te\-re\-shchen\-kivska st. 3, Kyiv, 01601 Ukraine}
\email{kaveh@imath.kiev.ua}

\makeatletter
\@addtoreset{equation}{section}
\@addtoreset{figure}{section}
\@addtoreset{table}{section}
\makeatother

\newtheorem{theorem}{Theorem}[section]
\newtheorem{lemma}{Lemma}[section]
\newtheorem{remk}{Remark}[section]
\newtheorem{prop}{Proposition}[section]
\newtheorem{defn}{Definition}[section]
\newtheorem{cor}{Corollary}[section]
\theoremstyle{definition}

\DeclareMathAlphabet{\mathpzc}{OT1}{pzc}{m}{it}

\newcommand{\pr}{\Pi}

\DeclareMathOperator{\Ind}{Ind}

\begin{document}

\begin{abstract}
In this paper we develop the geometry of bounded Fr\'{e}chet manifolds. We prove that a bounded Fr\'{e}chet tangent bundle admits a vector bundle structure. 
But the second order tangent  bundle  $T^2M$ of a bounded Fr\'{e}chet manifold $M$, becomes a vector bundle over $M$ if and only if $M$ is endowed with a 
linear connection. As an application, we prove  the existence and uniqueness of the integral curve of a vector field on $M$.
\end{abstract}
\maketitle

\section{Introduction}
The geometry of Fr\'{e}chet manifolds has received serious attention in recent years, cf.~\cite{dd} for a survey. In particular, 
second order tangent bundles have been studied due to their applications in the study of second order ordinary differential
equations that arise via geometric objects (such as  autoparallel curves and  parallel translation) on manifolds (see~\cite{4},~\cite{sdgs}).
However, due to intrinsic difficulties with 
Fr\'{e}chet spaces only a certain type of manifolds was considered, namely those Fr\'{e}chet manifolds which can be obtained as projective limit of 
Banach manifolds (PLB-manifolds). It was proved that the second order tangent bundle $T^2M$ of a PLB-manifold $M$ admits a vector 
bundle structure if and only if $M$ is endowed with a linear connection (see~\cite{gd}). 

Some of the basic issues in the theory of Fr\'{e}chet spaces are mainly related with the space of continuous linear mappings. 
Indeed, the space of continuous linear mappings of one Fr\'{e}chet space to another is not a Fr\'{e}chet space in general.
On the other hand, the general linear group of a Fr\'{e}chet space does not admit any non-trivial topological group structure. This defect puts  in question 
the way of defining vector bundle. Another drawback is 
 the lack of a general solvability theory for ordinary differential equations. Because of these reasons, in the framework of Fr\'{e}chet
 bundles an arbitrary connection is hard to handle. 

As remarked, there is a way out of these difficulties for Fr\'{e}chet manifolds which can be obtained as projective limit of Banach manifolds.
However, there is another way to overcome aforementioned problems. Recently, in the suggestive paper~\cite{m}, M\"{u}ller introduced the concept of bounded Fr\'{e}chet manifolds
and provided an inverse function theorem in the sense of Nash and Moser in this category. Such spaces arise in geometry and physical field theory and have
many desirable properties. For instance, the space of all smooth sections of a fiber bundle (over closed or non-compact manifolds), which
is the foremost example of infinite dimensional manifolds, has the structure of a bounded Fr\'{e}chet manifold, see~\cite[Theorem 3.34]{m}. 
As for the importance of bounded Fr\'{e}chet manifolds, we refer
to the paper~\cite{k}, where Sard's theorem was obtained in this category. The statement of the theorem is as follows: Let $ M $ resp. $ N $ be
bounded Fr\'{e}chet manifolds with compatible metrics $ d_M $ resp. $ d_N$ modelled 
on Fr\'{e}chet spaces $ E$ resp. $F $ with standard metrics. Let $ f: M \rightarrow N $ be an 
$ MC^k$- Lipschitz Fredholm map with $  k > \max \lbrace {\Ind f,0} \rbrace $. Then the set of regular
values of $ f $ is residual in $ N $.

One of the essential ideas of this setting is to replace the space of all  continuous linear maps by the space $\mathcal{L}_{d',d}(E,F)$ of all linear Lipschitz continuous maps.
Then  $\mathcal{L}_{d',d}(E,F)$ is a topological group that has satisfactory properties. For example, the composition map 
$\mathcal{L}_{d,g}(F,G) \times \mathcal{L}_{d',d}(E,F) \rightarrow \mathcal{L}_{d',g}(E,G) $ is bilinear continuous. In particular,
the evaluation map $\mathcal{L}_{g,d}(E,F) \times E \rightarrow F$ is continuous. 

Our goal in this paper is to extend to bounded Fr\'{e}chet manifolds the known results of Fr\'{e}chet geometry. 
We define  the tangent bundles $TM$ and $T^2 M$ of a bounded Fr\'{e}chet manifold $M$ modelled on a Fr\'{e}chet space $F$ and prove that they too are endowed 
with bounded Fr\'{e}chet manifold structures of the same type modelled on $F^2$ and $F^4$, respectively. 
In addition, we show that $TM$ admits  a vector bundle structure, which allows us to define a connection on $TM$ via a connection map (cf.~\cite{vas},~\cite{va}).
We shall interpret linear connections as linear systems of ordinary 
differential equations on trivial bundles. Our main result is that  $T^2M$ admits
a vector bundle structure if and only if $M$ is endowed with a linear connection. Moreover, a linear connection
on $M$ determines a vector bundle structure on $T^2M$ and a vector bundle isomorphism  $T^2M \rightarrow TM \oplus TM$. We conclude by proving
the existence and uniqueness of  the integral curve of a vector field on $M$.

 It turns out that  bounded Fr\'{e}chet manifolds have some advantages over both PLB-manifolds and infinite dimensional convenient manifolds.  
In the case of PLB-manifolds, the difficulty is that to construct a geometric object on manifolds we need to establish 
the existence of the projective limit of its Banach corresponding factors. While in the case of convenient manifolds, to construct
geometrical structures on manifolds we need to define the notion of manifolds by charts but this restricts the consequences of 
Cartesian closedness  drastically (see~\cite{km},~\cite{mi}). In addition, for convenient manifolds we have two different kinds of tangent bundles 
(kinematic and operational) and hence we have two different types of vector fields. Another drawback  is that 
operational vector fields do not necessarily have integral curves. On the other hand, for a given kinematic vector field  integral curves may not exist 
locally, and if they exist they may not be unique for the same  initial condition (see~\cite{km}). 
\section{Prerequisites}  
In this section we summarize all the necessary preliminary material that we need
for a self contained presentation of the paper. For detailed studies on bounded Fr\'{e}chet manifolds we refer to~\cite{k},~\cite{glockner} and~\cite{m}. 

We denote by $(F,d)$ a Fr\'{e}chet space  whose 
topology is defined by a complete translational-invariant metric $d$. 
We define $  \parallel f \parallel_d \coloneq d(f,0) $ for $ f \in F $ and  write $ L.f $ instead of $ L(f) $ when $ L $ is a  linear map
between Fr\'{e}chet spaces. A metric with absolutely convex balls will be called a standard metric. Note that
every Fr\'{e}chet space  admits a standard metric which defines its topology: If $\alpha_n$ is an arbitrary sequence of 
positive real numbers converging to $zero$ and if $\rho_n$ is any sequence of continuous semi-norms defining the 
 topology of $F$. Then
\begin{equation*}
d_{\alpha,\, \rho}(e,f)\coloneq \sup_{n \in \mathbb{N}} \alpha _n \dfrac{\rho_n(e-f)}{1+\rho_n(e-f)}
\end{equation*}
is a metric on $F$ with the desired properties.

As mentioned in the Introduction, we replace the space of all linear continuous  maps between Fr\'{e}chet spaces by the space of all linear Lipschitz continuous
maps. Let $(E,g)$ be another Fr\'{e}chet space and let $\mathcal{L}_{g,d}(E,F)$ be the set of all globally linear Lipschitz continuous maps, i.e. linear maps $ L: E \rightarrow F $ such that 
\begin{equation*}
\parallel L \parallel_{g,d}\, \coloneq \displaystyle \sup_{x \in E\setminus\{0\}} \dfrac{\parallel L.x \parallel_d}{\parallel x \parallel_g} < \infty.
\end{equation*}
We abbreviate $\mathcal{L}_{g}(E)\coloneq \mathcal{L}_{g,g}(E,E)$ and write $\parallel L\parallel_{g}\, = \,\parallel L\parallel_{g,g}$ for $L \in \mathcal{L}_{g}(E) $.
 If $d$ is a standard metric, then
\begin{equation} \label{metric}  
 D_{g,d}: \mathcal{L}_{g,d}(E,F) \times \mathcal{L}_{g,d}(E,F) \longrightarrow [0,\infty) , \,\,
(L,H) \mapsto \parallel L-H \parallel_{g,d}
\end{equation}
is a translational-invariant metric on $\mathcal{L}_{d,g}(E,F)$ turning it into an Abelian topological group (see~\cite[Remark 1.9]{glockner}). The latter is not a topological vector space
in general, but a locally convex vector group with absolutely convex balls. We shall always equip Fr\'{e}chet spaces  with standard metrics and define the topology on $\mathcal{L}_{d,g}(E,F)$ by the  metric $D_{g,d}$.
The vector groups $\mathcal{L}_{g,d}^{(i+1)}(F,E) \coloneq (F,\mathcal{L}_{g,d}^i(F,E))$ are defined by induction. 

Let $ E,F $ be Fr\'{e}chet spaces, $ U $ an open subset of $ E $, and $ P:U \rightarrow F $
 a continuous map. Let $CL(E,F)$ be the space of all continuous linear maps from $E$ to $F$ topologized by the compact-open topology. 
 We say $ P $ is differentiable at the point $ p \in U$ if there exists a linear map
$\operatorname{d}P(p): E \rightarrow F$ with $\operatorname{d}P(p)h =  \lim_{t\rightarrow 0}\tfrac{P(p+th)-P(p)}{t}$, for all $ h \in E $. 
If $P$ is differentiable at all points $p \in U$, if $\operatorname{d}P(p) : U \rightarrow CL(E,F)$ is continuous for all
$p \in U$ and if the induced map $ P': U \times E \rightarrow F,\,(u,h) \mapsto \operatorname{d}P(u)h $
 is continuous in the product topology, then we say that $ P $ is Keller-differentiable.
  We define $ P^{(k+1)}: U \times E^{k+1} \rightarrow F $ inductively by
\begin{equation*}
P^{(k+1)}(u,f_{1},...,f_{k+1}) = \lim_{t\rightarrow 0}\dfrac{P^{(k)}(u+tf_{k+1})(f_1,...,f_k)- P^{(k)}(u)(f_1,...,f_k)}{t}.
\end{equation*}

If $P$ is Keller-differentiable, $ \operatorname{d}P(p) \in \mathcal{L}_{d,g}(E,F) $ for all $ p \in U $, and the induced map 
$ \operatorname{d}P(p) : U \rightarrow \mathcal{L}_{d,g}(E,F)   $ is continuous, then $ P $ is called b-differentiable. We say $ P $ is $ MC^{0} $ and write $ P^0 = P $ if it is continuous. 
We say $P$ is an $ MC^{1} $ and write  $P^{(1)} = P' $ if it is b-differentiable. Let $ \mathcal{L}_{d,g}(E,F)_0 $ be 
the connected component of $ \mathcal{L}_{d,g}(E,F) $ containing the zero map. If $ P $ is b-differentiable and if 
 $V \subseteq U$ is a connected open neighborhood of $x_0 \in U$, then $P'(V)$ is connected and hence contained in the connected component
$P'(x_0) +  \mathcal{L}_{d,g}(E,F)_0 $ of $P'(x_0)$ in $\mathcal{L}_{d,g}(E,F)$. Thus, $P'\mid_V - P'(x_0):V \rightarrow \mathcal{L}_{d,g}(E,F)_0 $
 is again a map between subsets of Fr\'{e}chet spaces. This enables a recursive definition: If $P$ is $MC^1$ and $V$ can be chosen for each
 $x_0 \in U$ such that $P'\mid_V - P'(x_0):V \rightarrow \mathcal{L}_{d,g}(E,F)_0 $ is $ MC^{k-1} $, 
then $ P $ is called an $ MC^k$-map. We make a piecewise definition of $P^{(k)}$ by $ P^{(k)}\mid_V \coloneq \left(P'\mid_V - P'(x_0)\right)^{(k-1)} $
for $x_0$ and $V$ as before. 
The map $ P $ is $ MC^{\infty} $ if it is $ MC^k $ for all $ k \in \mathbb{N}_0 $. We shall denote by $\operatorname{D},\operatorname{D^2}$  the first and the second differential, respectively.

A bounded Fr\'{e}chet manifold is a Hausdorff second countable topological space with an atlas of coordinate 
charts taking their values in Fr\'{e}chet spaces such that the coordinate transition functions are all
 $ MC^{\infty} $-maps. 
 
We will need to consider the space of all globally Lipschitz continuous $k$-multilinear maps. 
Let $B =\prod_{i=1}^{i=k} F_i$ be the topological product of any finite number $k$ of Fr\'{e}chet spaces $(F_1,d_1),\ldots,(F_k,d_k)$.
For $x=(x_1,\ldots,x_k) \in B$ and  $y=(y_1,\ldots,y_k) \in B$, we define
the maximum metric $d_{max}$ as follows: $d_{max}(x,y)= \max _{1 \leq i \leq k} d_i(x_i,y_i)$. We shall always use this metric on $B$.
Let $(F_1,d_1),\ldots,(F_k,d_k)$ and $(F,d)$ be Fr\'{e}chet spaces. The space of all globally Lipschitz continuous $k$-multilinear maps 
is the space of all $k$-multilinear maps $L: F_1\times   \ldots \times F_k \rightarrow F$ such that for all $f_i \in F_i \setminus\{0\}$,
$1\leq i \leq k$, 
\begin{equation*}
\parallel L \parallel_{d_1,\ldots,d_k,d} \,\,\, \coloneq \displaystyle \sup_{f_i \in F_i \setminus\{0\}}  \dfrac{\parallel L(f_1,\ldots,f_k) \parallel_{d}}{\parallel f_1 \parallel_{d_1}\ldots \parallel f_k \parallel_{d_k}}   < \infty.
\end{equation*}
This space is denoted by $\mathcal{L}_{d_1,\ldots d_k,d}(F_1 ,\ldots,F_k;F )$.
We define on the latter space a metric
\begin{equation*}
 D_{d_1,\ldots,d_k,d}(L,H) = \, \parallel L - H \parallel_{{d_1,\ldots,d_k,d}}
\end{equation*}
which makes it into an Abelian topological group. 

Throughout the paper, we suppose that $d_1,\ldots ,d_k,d$ are fixed metrics and we will not write them when they appear as indices in the notations
to make the notations more readable.

\paragraph{Convention.} The terms bounded Fr\'{e}chet tangent bundle and bounded Fr\'{e}chet second order tangent bundle are too long, so we remove 
 ``bounded Fr\'{e}chet'' from the terms.
\section{Constructions of $TM$ and $T^2M$}\label{tangs}
In this section we construct $TM$ and $T^2M$ based on the work  of Yano and Ishihara~\cite{ish}.
\subsection{ Tangent bundle.} 
 
Let $M$ be a bounded Fr\'{e}chet manifold modelled on a Fr\'{e}chet space $F$. Let $\mathcal{MC}_p(M)$ be the set of 
 all $MC^{\infty}$- mappings $\mathpzc{f}:\mathbb{R}\rightarrow M$ that send $zero$ to $p \in M$.
 We define on $\mathcal{MC}_p(M)$ an equivalence relation $\sim$ as follows: Let $\Phi=\{ (U_\alpha,\varphi_\alpha)\}_{\alpha \in \mathcal{A}}$
 be a compatible atlas for $M$, $(p\in U_{\alpha},\varphi_{\alpha})$ an admissible chart, and $\mathpzc{f},\,\mathpzc{g} \in \mathcal{MC}_p(M)$.
 Let $r$ be a fixed natural number.
 We say that $\mathpzc{f}$ and $\mathpzc{g}$ are equivalent and write $\mathpzc{f}\sim \mathpzc{g}$ if they satisfy the following:
 \begin{equation} \label{eq}
  (\varphi_{\alpha}\circ \mathpzc{f})'(0) = (\varphi_{\alpha} \circ \mathpzc{g})'(0),\cdots  , (\varphi_{\alpha} \circ \mathpzc{f})^r(0) = (\varphi_{\alpha} \circ \mathpzc{g})^r(0),
 \end{equation}
where the orders of the derivatives run between 1 and $r$. 
It follows from the chain rule for $MC^k$-maps (see~\cite[Lemma B.1 (f)]{glockner})
that the equivalency  at a point $p$ is well defined. The equivalence class containing a mapping
$\mathpzc{f} \in  \mathcal{MC}_p(M)$ is called the $r$-jet of $\mathpzc{f}$ at $p$ and is denoted by $j^r_p \mathpzc{f}$.

 Let $TM$ be the set of all 1-jets of $M$ and let $\pi_M: TM \rightarrow M$ be a natural projection. The fiber $ \pi^{-1}_M(p)$ is the tangent space $T_pM$.
 The space $T_pM$ has the structure of a Fr\'{e}chet space which is isomorphic to $F$ by means of the mapping $\varphi_{\alpha} \circ \pi_M : T_pF \rightarrow F$ given by
 $j^1_p\mathpzc{f} \mapsto \varphi_{\alpha}(p)$. It is easily verified that this structure of $T_pM$ is independent of the choice of the chart $(U_{\alpha},\varphi_{\alpha})$. 
 Then $TM$ is the disjoint union of the tangent spaces $T_pM$ and is called the tangent bundle over $M$.  
 Let $h: M \rightarrow N$ be an $MC^k$-map of manifolds. 
 The tangent map $Th :TM \rightarrow TN$ is defined by $T h(j^1_p(\mathpzc{f}))  = j^1_{h(p)}(h\circ \mathpzc{f})$.

The following lemma  is fundamental for constructing  trivializing atlases and  vector bundle structures for $TM$ and $T^2M
$. 
\begin{lemma}\label{bundiso}
\begin{description}
 \item[(i)] Let $h: M \rightarrow N$ and $g : N \rightarrow K$ be $MC^k$-maps of manifolds. Then $T (h \circ g) = Tg \circ Th$. 
 \item[(ii)] If $h: M \rightarrow N$ is an $MC^k$-diffeomorphism, then $Th:TM \rightarrow TN$ is a bijection and $(Th)^{-1}= T(h^{-1})$.
 \item [(iii)] Let $h : U \subset E \rightarrow V \subset F$ be a diffeomorphism of open sets of Fr\'{e}chet spaces. The tangent map $Th: U \times F \rightarrow V\times E $ is a local vector bundle isomorphism.
 \item [(iv)]\label{lk} If $h :U \subset E \rightarrow V \subset F$ is an $MC^k$-diffeomorphism of open sets of Fr\'{e}chet spaces, then  $Th$ is an $MC^{k-1}$-diffeomorphism.
\end{description}
\end{lemma}
\begin{proof}
\begin{description}
 \item[(i)] $g \circ h$ is $MC^k$ (\cite[Lemma B.1 (f)]{glockner}). Furthermore,
 \begin{equation*}
  T(g \circ h)(j^1_p \mathpzc{f}) = j^1_{(g\circ h)(p)}(g\circ h \circ \mathpzc{f}) = Tg(j^1_{h(p)})( h \circ \mathpzc{f}) = (Tg \circ Th)(j^1_p \mathpzc{f}).
 \end{equation*}
 \item[(ii)] By the previous part and the definition of the tangent map, $Th \circ Th^{-1} = T \mathrm{id}_{TN}$ while $Th^{-1} \circ Th = T \mathrm{id}_{TM}$.
 \item[(iii)] $Th$ is a local vector bundle morphism. Since $h$ is a diffeomorphism it follows that $(Th)^{-1}= T(h^{-1})$ is a local vector bundle morphism, thus $Th$ is
 a vector bundle isomorphism.
 \item[(iv)] Let $C$ be a curve passing through $u \in U$ such that $\operatorname{D}C(0)\cdotp 1 = e$ for a given $e \in F$. Define the map $\eta(t):\mathbb{R} \rightarrow E$  by $\eta(t) = u + e t$ which
 is tangent to $C$ at $t = 0 $. 
 Define $\lambda : U \times F \rightarrow TU$ by $\lambda (u,e) = j^1_u(\eta(t))$.
 We have $(Th \circ \lambda)(u,e) = Th\cdotp j^1_u (\eta(t)) = j^1_{h(u)}(h \circ \eta(t))$. Also we have $(\lambda \circ h')(u,e) = \lambda (h(u), \operatorname {D}h(u)\cdotp e) = j^1_{h(u)}(h(u)+ (\operatorname {D}h(u)\ldotp e)t)$.
These are equal because the curves $t \mapsto h (u + et)$ and $t \mapsto h(u) +  (\operatorname {D}h(u)\ldotp e)t$ are tangent at $0$ by the definition
of the derivative and the previous parts. Therefore, $Th \circ \lambda = \lambda \circ h'$, which means
  $\lambda$ identifies $U \times E$ with $TU$. Correspondingly, we can identify $h'$ with $Th$, so the results of earlier parts imply statement (iv).
 \end{description}
\end{proof}
\begin{prop}
 Let $\pi_M : TM \rightarrow M$ be a tangent bundle. Then the atlas $\{ (U_\alpha,\varphi_\alpha)\}_{\alpha \in \mathcal{A}}$ gives rise to a trivializing atlas
$\{ (\pi_M^{-1}(U_\alpha),T\varphi_\alpha)\}_{\alpha \in \mathcal{A}}$ on $TM$, with
\begin{equation*}
 T\varphi_\alpha : \pi_M^{-1}(U_\alpha) \rightarrow \varphi_\alpha (U_\alpha)\times F , \quad j^1_p(\mathpzc{f}) \mapsto (\varphi_\alpha(p), (\varphi_\alpha \circ \mathpzc{f})'(0));\,\,\mathpzc{f} \in \mathcal{MC}_p(M).
\end{equation*}
This makes $TM$ into a bounded Fr\'{e}chet manifold modelled on $ F \times F$.
\end{prop}
\begin{proof}
It follows from Lemma~\ref{bundiso}. 
\end{proof}
We will apply the definition of differentiable vector bundles due to Neeb \cite{neeb}. We will need the following notion of differentiability which is the adaption of the differentiability given for Keller $ C_c^k$-maps in \cite[Definition II.3.1]{neeb}. 
\begin{defn}\label{def:neeb}
	Let $ {M} $ be an $ MC^k,\, (k\geq 1) $ Fr\'{e}chet manifold, and $ \mathrm{Diff}({M}) $ the group of
	diffeomorphisms of $ {M} $. Further, let $ {N} $ be an $ MC^k$  Fr\'{e}chet manifold. Although, in general, $ \mathrm{Diff}({M} ) $ has
	no natural Lie group structure,  a map $ \varphi: {N} \to \mathrm{Diff}({M})$ is said to be $ MC^k$, if the following map 	is of class $ MC^k$:
	\begin{equation*}
	\widehat{\varphi} : {N} \times {M} \to {M} \times {M}, \quad (n,x) \mapsto (\varphi(n)(x),\varphi^{-1}(n)(x))
	\end{equation*}
\end{defn}
\begin{defn}
Let ${M}$ be an $ MC^k $-Fr\'{e}chet manifold modeled on a Fr\'{e}chet space ${F}$, $ k\geq1 $, and $ {E} $ another Fr\'{e}chet space.	
A $MC^r$-{vector bundle} of type ${E}$ over ${M}$ is a triple $(\pr, {V}, {E})$, consisting of a $MC^r$-manifold $ {V} $, a $ C^r $-map $ \pr: {V} \to {M} $ and a Fr\'{e}chet space $ {E} $, with the following properties:
\begin{description}
	\item[(VB.1)]  $ \forall m \in {M} $, the fiber $ {V}_m := \pr^{-1}(m) $ is a Fr\'{e}chet space 
	isomorphic to ${E} $.
	\item[(VB.2)] Each  $ m \in {M} $ has an open neighborhood $ U $ for which there exists a diffeomorphism
	\begin{equation*}
	\phi_U:\pr^{-1}(U) \to U \times {E} 
	\end{equation*}
	with $ \phi_U = (\pr \rvert_U, \psi_U) $, where $ \psi_U  : \pr^{-1}(U) \to {E} $ is linear on each  $ {V}_m, m \in U$.		
\end{description}
We then call $ U $ a trivializing subset of $ {M} $ and $ \phi_U $ a bundle chart. If $ \phi_U $ and $ \phi_V $ are two bundle charts and $ U \cap V \neq \emptyset $, then we obtain a diffeomorphism
\begin{equation*}
\phi_U \circ  \phi_V^{-1}: \phi_V(U \cap V) \times {E} \to \phi_U(U \cap V) \times {E}
\end{equation*}
of the form $ (x, v) \mapsto (x, \psi_{VU} (x)v) $. This leads to a map
$ \psi_{UV} : U \cap V \to \mathbf{GL}({E}) $
for which it does not make sense to speak about smoothness because $ \mathbf{GL}({E}) $ is not a Lie group. Nevertheless, $ \psi_{UV} $ is of class $ C^r $ in the sense (Definition \ref{def:neeb}) that the map
\begin{gather*}
\widehat{\psi_{UV}}: (U \cap V) \times {E} \to (U \cap V) \times {E} \\
(x,v) \mapsto
(\psi_{UV}(x)v,\psi_{UV}(x)^{-1}v) =(\psi_{UV}(x)v,\psi_{VU}(x)v)
\end{gather*}
is of class $ C^r $. Here, $ \mathbf{GL}({E}) $ is the general linear group of $ {E} $.
\end{defn}
\begin{theorem}
Let ${M}$ be an $ MC^k $-Fr\'{e}chet manifold modeled on a Fr\'{e}chet space ${F}$, $ k\geq1 $.	
 $TM$ admits a vector bundle structure over $M$ with fiber of type $F$.
\end{theorem}
\begin{proof}
 Consider the above atlas of $M$ and its corresponding trivializing atlas for $TM$. Let $\operatorname{\pi r_1},\operatorname{\pi r_2}$ be the projections
 to the first and the second factors, respectively. For all $\alpha \in \mathcal{A}$ we have $ \operatorname{\pi r_1} \circ T\varphi_{\alpha} = \pi_M$, therefore $TM$ is a fiber bundle. Suppose $ U_{\alpha} \cap U_{\beta} \neq 0$, then
 by Lemma~\ref{bundiso} (iii) the overlap map
 \begin{equation*}
  T\varphi_{\alpha} \circ T\varphi_{\beta}^{-1}: \varphi_{\beta} (U_{\alpha} \cap U_{\beta}) \times F \rightarrow  \varphi_{\alpha} (U_{\alpha} \cap U_{\beta}) \times F
 \end{equation*}
is a local vector bundle isomorphism.  Let $\Theta_{\alpha\beta} = T\varphi_\alpha \circ T\varphi_{\beta}^{-1}$ be the transition map. 
The following map
\begin{gather*}
\widehat{\Theta_{\alpha\beta}}: \varphi_{\beta}(U \cap V) \times {E} \to \varphi_{\alpha}(U \cap V) \times {E} \\
(x,v) \mapsto
(\Theta_{\alpha\beta}(x)v,\Theta_{\alpha\beta}(x)^{-1}v) =(\Theta_{\alpha\beta}(x)v,\Theta_{\beta\alpha}(x)v)
\end{gather*}
is an $ MC^{k-1} $ morphism. Thus, $TM$ is an  $ MC^{k-1} $ vector bundle over $M$.
\end{proof}
\subsection{Second order tangent bundle.}

Now that $TM$ is a manifold we can define second order tangents:  
Assume $r=2$ in the equivalence relation~\eqref{eq}. Let $T^2_pM $ be the set of all $2$-jets at $p$ and let $T^2M \coloneq \bigcup _{p \in M}T^2_pM$.
Let $\Pi_{TM} : T^2M \rightarrow M$ be a natural projection defined by $\Pi_{TM}(j^2_p(\mathpzc{f})) = p$. If we topologize $T^2M$ in a natural way,
then $T^2M$ is called the second order tangent bundle over $M$. 

By virtue of
Lemma~\ref{bundiso}, we have a trivializing atlas $\{(\Pi^{-1}_{TM}(\pi_M^{-1}(U_\alpha)),\widetilde{\varPhi}_\alpha)\}_{\alpha \in \mathcal{A}}$
for $T^2M $ with
\begin{equation*}
 \widetilde{\varPhi}_\alpha : \Pi^{-1}_{TM}(\pi_M^{-1}(U_\alpha)) \rightarrow \varphi_\alpha (U_\alpha)\times F ,\quad j^2_p(\mathpzc{f}) \mapsto (\varphi_\alpha(p), (\varphi_\alpha \circ \mathpzc{f})''(0)); \,\, \mathpzc{f} \in \mathcal{MC}_p(M).
\end{equation*}
$T^2_pM$ can be identified with $F \times F$ under the isomorphism:
\begin{equation*}
 \varPsi : T^2_p M \rightarrow F \times F, \quad j^2_p(\mathpzc{f}) \mapsto ((\varphi_\alpha \circ \mathpzc{f})'(0),(\varphi_\alpha \circ \mathpzc{f})''(0)),
\end{equation*}
but fails to be a vector bundle over $M$ because the trivializing isomorphism does not respect the linear structure of the fibers. The submersion $\pi_{12} : T^2M \rightarrow TM$
defined by $ \pi_{12}(j^2_p(\mathpzc{f})) = j^1_p(\mathpzc {f})$ is a vector bundle. Let $\pi_2 :T(TM)\rightarrow TM$ be an ordinary
tangent bundle over $TM$. The space $T^2M$ coincides with
\begin{equation} 
 \left\lbrace \Upsilon \in T(TM) \mid \pi_2 (\Upsilon) = T\pi_M (\Upsilon)\right\rbrace,
\end{equation} 
 and can be identified with a submanifold of $T(TM)$, see~\cite[Proposition 3.2, p. 372 ]{tan}. The bundle $T(TM)$ is a fiber bundle over $M$ with the 
 projection $\pi^2 = \pi_M \circ T\pi_M $. The restriction $\pi^2 \mid _{T^2M} :{T^2M} \rightarrow M$ is again a fiber bundle.

\section{Connection}
In this section we define connections by using Vilms~\cite{va} point of view for connections on infinite dimensional vector bundles. Also,
we show that each linear connection corresponds in a bijective way to an ordinary differential equation analogous to the case of Banach
manifolds (see~\cite{vas}).

 Henceforth, we keep the formalism of Section~\ref{tangs} for tangent bundles and second order tangent bundles.
\begin{defn}
A \emph{strengthened connection map} for $ T{M} $ is a map $ \mathcal{K}: T(T{M}) \to T{M}  $, which is fully
determined by its local form:
\begin{gather*}
\mathcal{K}_{\alpha}\coloneq \varphi_\alpha \circ \mathcal{K} \circ (\widetilde{\varPhi}_\alpha))^{-1},\\
 \varphi_\alpha(U_\alpha) \times {F} \times {F}\times  {F} \to \varphi_\alpha(U_\alpha) \times {F},\quad
 \mathbb{K}_{\alpha} = (f,g,h,k)=  (f,k+\tau_{\alpha}(f,g) h),
 \end{gather*}
 for a family of  mappings $$ \tau_{\alpha}: \varphi_\alpha(U_\alpha) \times {F}  \to \mathcal{L}_{{d}}({F})^{\times} .$$
 Here $ \mathcal{L}_{{d}}({F})^{\times} $ is a subset of $ \mathcal{L}_{{d}}({F}) $ consists of invertiable mappings.
 The mapping $ \tau_{\alpha} $ is $ MC^{k-1} $ in the sense that the map
 \begin{equation*}
 \widehat{\tau_{\alpha}}: (\varphi_\alpha(U_\alpha) \times {F}) \times {F} \to {F} \times {F},\quad
 (x,y,h) \mapsto (\tau_{\alpha}(x,y)(h), \tau_{\alpha}^{-1}(x,y)(h))
 \end{equation*}
 is $ MC^{k-1} $, see Definition \ref{def:neeb}. It follows of course that $ \mathbb{K} $ is of class $ MC^{k-1} $. 
\end{defn}
\begin{remk}
	In the case of Banach manifolds, it is not required that the maps $ \tau_{\alpha} $ be invertible. However, we require that they be invertible to compensate for  the fact that $ \mathcal{L}_{{F}}({F}) $ is not a Fr\'{e}chet manifold. 
\end{remk}
A connection on $M$ is a connection
map on the tangent bundle $\pi_M : TM \rightarrow M$.  
A connection $\mathcal{K}$ is linear if and only if it is linear on the fibers of the tangent map. Locally $T\pi$ is the map $U_{\alpha}\times F \times F \times F \rightarrow U_{\alpha} \times F$ 
defined by $T\pi(f,\xi,h,\gamma)= (f,h)$, hence locally its fibers are the spaces $\{f\}\times F \times \{h\} \times F$. Therefore, $\mathcal{K}$ is linear on these fibers
if and only if the maps $(g,k)\mapsto k+ \tau_{\alpha}(f,g)h$ are linear, and this means that the mappings $\tau_{\alpha}$ need to be linear
with respect to the second variable.

Assume that the connection $\mathcal{K}$ is linear and $f \in U_{\alpha}$. By the canonical isomorphism of Lemma~\ref{iso} to $\tau_{\alpha}(f,\ldotp) \in \mathcal{L}(F,\mathcal{L}(F,F)) \cong \mathcal{L}(F\times F;F)$
is associated the unique local Christoffel symbol  $\Gamma_{\alpha}(p) :\varphi_{\alpha}(U_{\alpha})\rightarrow \mathcal{L}(F\times F;F)$ satisfying $\Gamma_{\alpha}(p)(g,h) = \tau_{\alpha}(p,g)h$. 
Christoffel symbols satisfy the following compatibility condition (cf.~\cite{P}):
\begin{equation}\label{ci}
 \tau_{\alpha}\Big(\operatorname{D}\Theta_{\alpha \beta}(f)(g),\operatorname{D}\Theta_{\alpha \beta}(f)(h)\Big)(\Theta_{\alpha \beta}(f)) +
 (\operatorname{D^2}\Theta_{\alpha \beta}(f)(h))(g) = \operatorname{D}\Theta_{\alpha \beta}(f)(\tau_\beta(g,h)(f))
\end{equation}
 for all $(f,g,h )\in \varphi_\alpha (U_\alpha \cap U_\beta)\times F\times F$. Here, $\Theta_{\alpha \beta}\coloneq \varphi_{\alpha} \circ \varphi^{-1}_{\beta} $. 

 \begin{theorem}\label{tang1}
  Every linear connection on $M$ induces a vector bundle structure on   $\pi^2 \mid _{T^2M} :{T^2M} \rightarrow M$
  and gives rise to an isomorphism of this vector bundle with the vector bundle $TM\oplus TM$.
 \end{theorem}
 \begin{proof}
 If we have a connection, then the connection map $\mathcal{K} : T(TM) \rightarrow M$ is defined. The following map
 \begin{equation}\label{di}
  \pi_2 \oplus \mathcal{K} \oplus T\pi_M : T(TM) \rightarrow TM \oplus TM \oplus TM
 \end{equation}
is a diffeomorphism (see~\cite{do}). The diffeomorphism determines a unique vector bundle structure for $T(TM)$ over $M$. 
Let $(U_{\alpha}, \varphi_{\alpha})$ be a chart of $M$. The induced chart $\{ (\pi_M^{-1}(U_\alpha),T\varphi_\alpha)\}$ in $TM$ takes
a vector bundle structure by means of the Diffeomorphism~\eqref{di}. Let $\imath : TM \oplus TM \rightarrow TM \oplus TM \oplus TM$ 
be the natural isomorphism. $T^2M$ is a submanifold of $T(TM)$ consisting of tangent vectors $\Upsilon$ such that $ \pi_2 (\Upsilon) = T\pi_M (\Upsilon)$. Therefore,
the inclusion $\imath$ is the isomorphism onto $(\pi_2 \oplus \mathcal{K} \oplus T\pi_M )(T^2M)$, thus
 \begin{equation*}
  \imath^{-1} \circ (\pi_2 \oplus \mathcal{K} \oplus T\pi_M )(T^2M) = \pi_2 \oplus \mathcal{K}(T^2M).
 \end{equation*}
Hence the diffeomorphism
\begin{equation}
 \pi_2 \oplus \mathcal{K} : T^2 M\rightarrow TM \oplus TM
\end{equation}
gives the structure of a vector bundle to $T^2M$. Since $T^2M$ is isomorphic to $TM \oplus TM$, it can be considered as a vector bundle with  group structure
$\operatorname {Aut}( F \times F)$. 
 \end{proof}
The proof of the following theorem is the same as the usual proof given for Banach manifolds (see~\cite[Theorem 2.4]{gd}). We just give literally the scheme of proof.   
 \begin{theorem}\label{tang2} 
  If $T^2M$ admits a vector bundle structure isomorphic to $TM\oplus TM$, then there exists a linear connection on $M$.
 \end{theorem}
\begin{proof}
 Let $\left\lbrace(\varPi^{-1}(U_{\alpha}), \Omega_{\alpha})\right\rbrace_{\alpha \in \mathcal{A}}$ be a trivializing atlas of $T^2M$. By hypothesis 
 $\Omega_{\alpha,p} = \Omega_{\alpha,p}^{1} \times \Omega_{\alpha,p}^2$, where $\Omega_{\alpha,p}^{i} : \pi_M^{-1}(p) \rightarrow F$ (i=1,2).
 Let $(U,\varOmega)$ be an arbitrary chart such that $U \subseteq U_{\alpha}$. Define 
 $\varOmega_{\alpha} = \varOmega \circ (\Omega_{\alpha,p}^{1} \circ(\operatorname{D}_x \varOmega)^{-1})$. Then define the mappings as follows:
 \begin{equation*}
\tau_{\alpha}(u,u)(y)  = \Omega_{\alpha,p}^2 (j^2_p \mathpzc{f}) - (\varOmega_{\alpha} \circ \mathpzc{f})''(0), \quad y \in \varOmega_{\alpha}({U_{\alpha})},
 \end{equation*}
where $\mathpzc{f}$ is the representative of the vector $u$.  The remaining values of $\tau_{\alpha}(y)$ on elements of the form 
$(u,v)$ with $u \neq v$ are automatically defined if we demand $\tau_{\alpha}(y)$ to be symmetric and bilinear. They satisfy the compatibility condition since
 the trivializations $\left\lbrace(\varPi^{-1}(U_{\alpha}), \Omega_{\alpha})\right\rbrace_{\alpha \in \mathcal{A}}$ coincide 
 on all common areas of their domains, hence give rise to a linear connection on $M$.
\end{proof}

 \section{Vector fields on $TM$}
Having introduced the tangent bundle over a manifold $M$, we now consider sections of these bundles. 
A vector field on $M$ is a section $\xi :M \rightarrow TM$ of its tangent bundle, i.e. $\pi_M \circ \xi = \operatorname{id}_M$. 
For a vector field $\xi$ and a chart $U \subset M \xrightarrow {\varphi} \varphi(U) \subset F$, the principle part
$\xi_{\varphi} :\varphi (U) \rightarrow F$
of $\xi$ is defined by $\xi_{\varphi}(\varphi(p)) = \operatorname{pr_2} \circ T\varphi (\xi_p)$. Let $I$
be an open interval in $\mathbb{R}$ and let $\ell : I  \rightarrow M$ be a curve  passing through $p_0$.
If $\xi$ is a vector field on $M$ and if $\xi_{\varphi}$ denotes the principle part of its local representative in a chart $\varphi$, then $\ell(t)$ is called
an integral curve of $\xi$ when $(\varphi \circ \ell)'(t) = \xi_{\varphi}(\varphi \circ \ell(t))$ for each $t$,
where $\varphi \circ \ell$ is the local representative of the curve $\ell$.
Note that if the base manifold $M$ is a Fr\'{e}chet space $F$ with differential structure induced by the chart $(F, \operatorname{id}_F)$,
then the above condition reduces to: 
$\ell'(t) = \operatorname{D}{\ell(t)}(1_\mathbb{R})$. 
That is, our definition is a natural generalization of the notion of derivative on a manifold $M$.
\begin{prop} \label{ht}
 Let $U \subseteq F$ be open and let $\xi : U \rightarrow F$ be $MC^k, k\geq 1$. Then for $p_0 \in U$, there is an integral curve
 $\ell : I \rightarrow F$ at $p_0$. Furthermore, any two such curves are equal on the intersection of their domains. 
\end{prop}
\begin{proof}
 Since $\xi$ is $MC^k$, it is bounded, say by $R$. Let $L$ be a positive real number. Pick a positive real number $r$  such that  $\overline{B_r(p_0)} \subseteq U$
 and $\parallel \xi (p) \parallel_d \leq L$ for all $p \in \overline{B_r(p_0)} $. Let $m = \min \{1/R, r/L\}$ and let $t_0$ be a real number.
 We shall show that there is a unique $MC^1$-curve $\ell(t)$, $t \in [t_0 - m, t_0 + m]$ whose image lies in $\overline{B_r(p_0)}$ and
 that satisfies
 \begin{equation} \label{piz}
  \ell'(t) = \xi (\ell (t)), \quad \ell (t_0) = p_0.
 \end{equation}
The conditions $\ell'(t) = \xi (\ell (t)), \, \ell (t_0) = p_0$ are equivalent to the integral equation
\begin{equation}\label{eq:piz}
 \ell(t) = p_0 + \displaystyle \int_{t_0}^t \xi (\ell(u))du
\end{equation}
Now define $\ell_n(t)$ by induction
\begin{equation*}
 \ell_0(t)= p_0 , \quad \ell_{n+1}(t) = p_0 + \displaystyle \int_{t_0}^t \xi (\ell_{n}(u))du.
\end{equation*}
The estimation on the size of integral (see~\cite[Lemma 1.10]{glockner}) yields $\ell_n(t) \in \overline{B_r(p_0)}$ for all $n$ and $t \in [t_0 - m, t_0 + m]$.
Furthermore, 
\begin{equation*}
 \parallel \ell_{n+1}(t) - \ell_n (t)\parallel_d \,\leq \dfrac{LR^n}{(n+1)!} \mid t -t_0 \mid^{n+1}.
\end{equation*}
To see this, assume that
\begin{equation*}
\| \ell_{n}(t) - \ell_{n-1} (t)\|_{d} \,\leq \dfrac{LR^{n-1}}{(n)!} \mid t -t_0 \mid^{n}.
\end{equation*}
Then we estimate as follows: (again assuming that $ t \geq t_0  $ for simplicity)
\begin{align*}
\| \ell_{n+1}(t) - \ell_n (t)\|_{d} &= \left\| \displaystyle \int_{t_0}^t \xi (\ell_{n}(u)) - \xi (\ell_{n-1}(u))d{u} \right\|_{d} \\
&\leq \displaystyle \int_{t_0}^t R \|\xi (\ell_{n}(u)) - \xi(\ell_{n-1}(u)) \|_{d}d{u} \\
&\leq R\displaystyle \int_{t_0}^t \dfrac{LR^{n-1}}{n!}(u-t_0)^n  d{u} \\
&=\dfrac{LR^{n}}{(n+1)!}(u-t_0)^{n+1}.
\end{align*}
Thus, since
\begin{equation*}
\dfrac{LR^{n}}{(n+1)!}(u-t_0)^{n+1}\leq \dfrac{LR^{n}}{(n+1)!}(m)^{n+1}
\end{equation*}
and the series with these quantities as terms is convergent, we see, writing  $$ \| \ell_{n+p} - \ell_{n}\|_{d} $$ as a telescoping sum, that the functions $\ell_n$ 
form a uniformly Cauchy sequence and hence converge uniformly to a continuous
curve $\ell(t)$ satisfying~\eqref{eq:piz}.
Since $\ell(t)$ is continuous, the integral
equation in fact shows that it is $ MC^{1} $. This proves existence.
Now let $\jmath(t)$ be another solution.
By a similar induction argument as above, we find that	
\begin{equation*}
\|  \ell_n (t) - \jmath (t)\|_{d} \,\leq \dfrac{LR^n}{(n+1)!} \mid t -t_0 \mid^{n+1}.
\end{equation*}
Therefore, letting $n \rightarrow \infty$ gives $\ell(t) = \jmath(t)$.

\end{proof}  
\begin{cor}\label{s}
  Suppose the hypotheses of the previous proposition hold. Let $\mathcal{I}_t(p_0)$ be the solution of $\ell'(t) = \xi (\ell (t)), \, \ell (t_0) = p_0$.
  Then there is an open neighborhood $U_0$ of $p_0$ and a positive real number $\alpha$  such that for every $q \in U_0$ there exists a unique integral
  curve $\ell(t) = \mathcal{I}_t(q)$ satisfying  $\ell(0) = q$ and $\ell'(t) = \xi (\ell (t))$ for all $t \in (-\alpha,\alpha)$.
\end{cor}

\begin{proof}
 Suppose $U_0 = \overline{B_{r/2}(p_0)}$ and $\alpha = \min \{1/R, r/2L\}  $. Fix an arbitrary point $q_0$  in $U_0$. Then $\overline{B_{r/2}(q_0)} \subset \overline{B_r (p_0)}$,
 thereby $\parallel \xi(z) \parallel_d < L$ for all $z \in \overline{B_{r/2}(q_0)}$. By Proposition~\ref{ht} with $p_0$ replaced by $q$, $r$ by $r/2$, and
 $t_0$ by $ 0$,
 for all $t \in (-\alpha,\alpha)\ $there exists a unique  integral curve $\ell(t)$ such that $\ell(0)=q$. 
\end{proof}
The proof of the following theorem is the same as the usual proof given for Banach manifolds (see~\cite[Theorem 2.1]{l}).
\begin{theorem}
 Let $\xi :M \rightarrow TM$ be a vector field. Then  there exits an integral curve for $\xi$ at $p \in M$. Furthermore, any two such curves are equal on the intersection of their domains.
\end{theorem}
\begin{proof}
The existence follows from Proposition~\ref{ht} by means of local representation. But that is not applicable for the proof of uniqueness since these curves may
lie in different charts. Let $\rho_i(t) : I_i \rightarrow M$ (i=1,2) be two integral curves. Let $I = I_1 \cap I_2$ and  $J = \{ t \in I \mid \rho_1(t) = \rho_2 (t) \}$.
$J$ is closed since $M$ is Hausdorff. From Proposition~\ref{ht}, $J$ contains some neighborhood of 0. 
Now define $\delta_1(u) = \rho_1 (u + t )$ and
 $\delta_2 (u)= \rho_2 (u + t)$ for $t \in J$. They are integral curves  with initial conditions $\rho_1(t)$
 and $\rho_2(t)$, respectively. By Proposition~\ref{ht} they coincide on a some neighborhood of 0. Therefore, $J$ contains an open neighborhood of $t$, so 
 $J$ is open. Since $I$ is connected it follows that $J = I$. 
 
\end{proof}

\end{document}